\documentclass[preprint]{elsarticle}

\usepackage{lineno,hyperref}
\modulolinenumbers[5]

\usepackage{amsmath,amssymb,amsthm}
\usepackage{setspace}
\usepackage{moreverb}
\usepackage{graphicx}
\journal{Journal of Algebra}

\newtheorem{theorem}{Theorem}[section]
\newtheorem{lemma}[theorem]{Lemma}
\newtheorem{corollary}[theorem]{Corollary}
\newtheorem{proposition}[theorem]{Proposition}
\newtheorem{algorithm}[theorem]{Algorithm}

\theoremstyle{definition}
\newtheorem{definition}[theorem]{Definition}
\newtheorem{example}[theorem]{Example}

\theoremstyle{remark}
\newtheorem{remark}[theorem]{Remark}

\numberwithin{equation}{section}

\input xy
\xyoption{all}

\newcommand\kk{\Bbbk}
\newcommand{\bZ}{{\mathbb{Z}}}
\newcommand{\lra}{{\longrightarrow}}
\newcommand{\cA}{{\mathcal A}}
\newcommand{\dm}{\mathrm {dim }}
\newcommand{\bR}{\mathbb{R}}
\newcommand{\lcm}{\mathrm {lcm\, }}

\begin{document}

\begin{frontmatter}

\title{Pruned cellular free resolutions of monomial ideals}

\author{Josep \`Alvarez Montaner\fnref{footnoteJA}}
\address{Departament de Matem\`atiques,
Universitat Polit\`ecnica de Catalunya, Spain}
\fntext[footnoteJA]{Partially supported by the {\it Generalitat de Catalunya} grant SGR2017-932 and
the Spanish {\it Ministerio de Econom\'ia y Competitividad} grant MTM2015-69135-P. He is a member of
the Barcelona Graduate School of Mathematics (BGSMath).}
\ead{Josep.Alvarez@upc.es}

\author{Oscar Fern\'andez-Ramos}
%\ead{caribefresno@gmail.com}

\author{Philippe Gimenez\fnref{footnotePG}\corref{mycorrespondingauthor}}
\address{
Instituto de Investigaci\'on en Matem\'aticas de Valladolid (IMUVA),
Universidad de Valladolid, Spain}
\fntext[footnotePG]{Partially supported by the Spanish
{\it Ministerio de Econom\'ia y Competitividad} grant MTM2016-78881-P
and {\it Consejer\'{\i}a de Educaci\'on de la Junta de Castilla y Le\'on} grant VA128G18.}
\ead{pgimenez@agt.uva.es}
\cortext[mycorrespondingauthor]{Corresponding author}

\begin{abstract}
Using discrete Morse theory, we give an algorithm that prunes the excess of information
in the Taylor resolution and constructs a new cellular free resolution for an arbitrary monomial ideal.
The pruned resolution is not simplicial in general, but we can slightly modify our algorithm in order to obtain
a simplicial resolution. We also show that the Lyubeznik resolution fits into our
pruning strategy.
The pruned resolution is not always minimal but it is a lot closer to the minimal resolution
than the Taylor and the Lyubeznik resolutions as we will
see in some examples.
We finally use our methods to give a different approach to the theory of splitting of monomial ideals.
We deduce from this splitting strategy that the pruned resolution is always minimal in the case
of edge ideals of paths and cycles.
\end{abstract}


\end{frontmatter}

%\linenumbers

\section{Introduction}

Let $R=\kk[x_1,\dots,x_n]$ be the polynomial ring over a field $\kk$ and $I\subseteq R$ a monomial ideal.
The study of minimal free resolutions of these ideals has been a very active area of research
during the last decades. There are topological and combinatorial formulae, as those of Hochster \cite{Hoc}
or Gasharov, Peeva and Welker \cite{GPW}, to describe their multigraded Betti numbers but, except for some
specific classes  of monomial ideals (see, e.g., \cite{EK}, \cite{BPS} or \cite{MSY}), the problem of describing a minimal
multigraded free resolution explicitly has shown to be difficult.

\vskip 2mm

Another strategy is to study non-minimal free resolutions.
These reveal less information than minimal free resolutions do but are
often much easier to describe. The most significant ones, that we will also comment on in this paper,
are the Taylor resolution \cite{Tay66} and
the Lyubeznik resolution \cite{Lyu88}, but one
should also mention the Scarf resolution of arbitrary monomial ideals obtained by deformation of exponents \cite{BPS}
and the hull resolution \cite{BS}.
An interesting feature of the Taylor and the Lyubeznik
resolutions is that they fit in the theory of simplicial resolutions
introduced by Bayer, Peeva and Sturmfels in \cite{BPS} and further extended
to regular cellular resolutions and CW-resolutions in \cite{BS} and \cite{JW} respectively. The idea behind
these three concepts is to associate to a free resolution of a monomial ideal a
simplicial complex (respectively a regular cell complex, a CW-complex) that carries in
its structure the algebraic structure of the free resolution. It is worth pointing
out that Velasco proved in \cite{Vel} that there exist monomial
ideals whose minimal free resolutions cannot be described by a CW-complex.

\vskip 2mm

By adapting the discrete Morse theory developed by Forman \cite{For} and
Chari \cite{CH},  Batzies and Welker provided in \cite{BW} a method to reduce
a given regular cellular resolution. In particular, they proved that the Lyubeznik
resolution can be obtained in this way from the Taylor resolution.
Let's point out that discrete Morse theory has the inconvenient fact that it can't be used
iteratively. To overcome this issue, one can use the algebraic discrete Morse theory developed independently by
 Sk\"oldberg \cite{Sko} and J\"ollenbeck and Welker \cite{JW}.
In this work, we use a similar strategy to reduce the Taylor
resolution and obtain cellular and simplicial free resolutions that
are closer to the minimal one than the Lyubeznik
resolution. Essentially, the information given by the Taylor
resolution can be encoded in a directed graph and the obstruction to
its minimality can be observed in some of the edges of this graph.
What we will do is to remove, in a convenient order, some of these
edges to provide a smaller resolution. In some sense, we are pruning
the excess of information given by the Taylor resolution in a simple and efficient way.

\vskip 2mm

The organization of this paper is as follows. In Section \ref{cellular},
we review the notion of  cellular resolution and introduce the basics
on discrete Morse theory that will be needed throughout this work.
In Section \ref{prune}, we present our main results. We first provide an algorithm
(Algorithm \ref{alg1}) that, starting from the Taylor resolution of a monomial ideal,
allows to construct a smaller cellular free resolution
(Theorem \ref{T1}).
The resolution that we obtain is not simplicial in general, but
we can adapt our pruning algorithm to produce a simplicial free resolution (Algorithm \ref{alg3}).
Indeed, the Lyubeznik resolution fits into this pruning strategy as shown in Algorithm \ref{alg2}.
Other variants of our method are also mentioned.

\vskip 2mm

In Section \ref{examples}, we illustrate our results with several examples.
We implemented our algorithms using CoCoALib \cite{cocoa}
for constructing pruned resolutions in the non-trivial examples contained in this section.
Finally, in Section \ref{split} we present a connection between our method and the
theory of Betti splittings introduced by Eliahou and Kervaire \cite{EK} and
later developed by Francisco, H\`a and Van Tuyl \cite{FHV}. We provide a sufficient condition
for having a Betti splitting by checking some prunings in our algorithm. We use this
approach to prove that the pruned resolution is minimal for edge ideals associated to paths and cycles.

\section{Cellular resolutions using discrete Morse theory} \label{cellular}

Let $R=\kk[x_1, \ldots, x_n]$ be the polynomial ring in $n$ variables with coefficients
in a field $\kk$. An ideal $I\subseteq R$ is monomial if it may be generated
by monomials ${\bf x}^\alpha:=x_1^{\alpha_1}\cdots x_n^{\alpha_n}$, where $\alpha \in \bZ_{\geq 0}^n$.
As usual, we set $|\alpha|:= {\alpha_1}+\cdots +{\alpha_n}$. Moreover, given a monomial set of generators of an ideal $I$, $\{m_1,\dots,m_r\}$,
we will consider the monomials $m_{\sigma}:={\rm lcm}(m_i \hskip 2mm | \hskip 2mm \sigma_i=1)$ for any $\sigma \in \{0,1\}^r$.
Denote by $\varepsilon_1,\dots,\varepsilon_r$ the standard basis of $\bZ^r$.

\vskip 2mm

A $\bZ^n$-graded free resolution of $R/I$ is an exact sequence of free $\bZ^n$-graded
modules:
\begin{equation}\label{resolution of I}
\mathbb{F}_{\bullet}: \hskip 3mm \xymatrix{ 0 \ar[r]& F_{p}
\ar[r]^{d_{n}}& \cdots \ar[r]& F_1 \ar[r]^{d_1}& F_{0} \ar[r]& R/I
\ar[r]& 0},
\end{equation}
where the $i$-th term is of the form
$$F_i =\bigoplus_{\alpha \in \bZ^n} R(-\alpha)^{\beta_{i,\alpha}}\,.$$
We say that $\mathbb{F}_{\bullet}$ is minimal if the matrices of the
homogeneous morphisms $d_i: F_i\longrightarrow F_{i-1}$ do not contain
invertible elements. In this case, the exponents $\beta_{i,\alpha}$ form a
set of invariants of $R/I$ known as its {\it multigraded Betti numbers}.
Throughout this work, we will mainly consider the coarser $\bZ$-graded free resolution. In this case, we
will encode the $\bZ$-graded Betti numbers in the so-called  {\it Betti diagram} of $R/I$ where the entry
on the $i$th column and $j$th row of the table is $ \beta_{i,i+j}$:

\begin{center}
 $\begin{matrix}
&0&1&2&\cdots \\ \text{total:}& \text{.} &\text{.}&\text{.}& \\
\text{0:} & \beta_{0,0} & \beta_{1,1} & \beta_{2,2} & \cdots \\\text{1:}& \beta_{0,1} & \beta_{1,2} & \beta_{2,3} & \cdots \\
\vdots& \vdots & \vdots & \vdots \\\end{matrix}
$
\end{center}

\vskip 2mm

\subsection{Cellular resolutions}

A CW-complex $X$ is a topological space obtained by attaching cells of increasing dimensions
to a discrete
set of points $X^{(0)}$.  Let $X^{(i)}$ denote the set of $i$-cells of  $X$ and
consider the set of all cells $X^{(\ast)}:=\bigcup_{i\geq 0} X^{(i)}$. Then, we can view $X^{(\ast)}$
as a poset with the partial order  given by $\sigma' \leq \sigma$ if and only if $\sigma'$
is contained in the closure of $\sigma$. We can also give a $\bZ^n$-graded structure to $X$ by means of an
order preserving map $gr: X^{(\ast)} \lra \bZ_{\geq 0}^n$.

\vskip 2mm

We say that the free resolution (\ref{resolution of I}) is {\it cellular} (or is a {\it CW-resolution})
if there exists a $\bZ^n$-graded CW-complex $(X, gr)$ such that, for all $i\geq 1$:

\begin{itemize}
 \item[$\cdot$]
 there exists a basis $\{e_\sigma\}$ of $F_i$ indexed by the $(i-1)$-cells of $X$, such that if
 $e_\sigma \in R(-\alpha)^{\beta_{i,\alpha}}$ then  $gr(\sigma)=\alpha$, and

 \item[$\cdot$]
the differential $d_i: F_i\longrightarrow F_{i-1}$ is given
by
 $$e_\sigma \hskip 2mm \mapsto \sum_{\sigma \geq \sigma' \in X^{(i-1)}} [\sigma: \sigma'] \hskip 1mm {\bf x}^{gr(\sigma) - gr(\sigma') } \hskip 1mm e_{\sigma'}\ ,\quad \forall \sigma \in X^{(i)}$$
where $[\sigma: \sigma']$ denotes the coefficient of $\sigma'$ in the image of $\sigma$ by the differential map in the cellular homology
of $X$.
\end{itemize}
In what follows, whenever we want to emphasize such a cellular structure, we will denote the free resolution as
$\mathbb{F}_{\bullet}=\mathbb{F}_{\bullet}^{(X,gr)}$. If $X$ is a simplicial complex, we say
that the free resolution is {\it simplicial}. This is the case for the following two well-known examples.

\vskip 2mm

$\bullet$ {\bf The Taylor resolution}:  The most recurrent example of simplicial free resolution is the Taylor resolution
discovered in \cite{Tay66}. Using the above terminology, we can describe it as follows.
Let $I=\langle m_1,\dots, m_r \rangle \subseteq R$ be a monomial ideal.
Consider the full simplicial complex on $r$ vertices, $X_{\tt Taylor}$, whose faces are labelled by
$\sigma \in \{0,1\}^r$ or, equivalently, by the corresponding monomials $m_\sigma$. We have a natural
$\bZ^n$-grading on $X_{\tt Taylor}$ by assigning $gr(\sigma)=\alpha \in \bZ^n$ where ${\bf x}^\alpha=m_\sigma$.
The {\it Taylor resolution} is the simplicial resolution $\mathbb{F}_{\bullet}^{(X_{\tt Taylor}, gr)}$.

\vskip 2mm

$\bullet$ {\bf The Lyubeznik resolution}:  Another important example of simplicial resolution is the one considered by Lyubeznik
in \cite{Lyu88}. Let's s start fixing an order $m_1 \leq \cdots \leq m_r$ on a generating set of a monomial ideal $I\subseteq R$.
Consider the simplicial subcomplex $X_{\tt Lyub} \subseteq X_{\tt Taylor}$ whose faces of dimension $s$ are labelled by
those $\sigma=\varepsilon_{i_0}+ \cdots + \varepsilon_{i_s} \in \{0,1\}^r$ with $i_0<\cdots <i_s$
such that, for all $t<s$ and all $j <i_t$,
$$ m_j \not | \hskip 2mm {\rm lcm}(m_{i_t},\dots , m_{i_s}).$$
The {\it Lyubeznik resolution} is the simplicial resolution $\mathbb{F}_{\bullet}^{(X_{\tt Lyub}, gr)}$.

\vskip 2mm

\subsection{Discrete Morse theory}
Forman introduced in  \cite{For} the discrete Morse theory as a method to
reduce the number of cells in a CW-complex without changing its
homotopy type. Batzies and Welker adapted this technique in \cite{BW}
to the study of cellular resolutions; see also \cite{Wel07}.  Indeed, they used the
reformulation of discrete Morse theory in terms of acyclic matchings
given by Chari in \cite{CH} in order to obtain, given a  regular  cellular
resolution (most notably the Taylor resolution), a reduced
 cellular resolution.

\vskip 2mm

This is also our approach in this work. Let's start recalling from \cite{BW} the preliminaries on discrete Morse theory.
Consider the directed graph $G_X$ on the set of cells of a regular $\bZ^n$-graded CW-complex $(X,gr)$ which edges are given by
$$E_X=\{ \sigma \lra \sigma' \hskip 2mm | \hskip 2mm  \sigma' \leq \sigma, \hskip 2mm   \dim \sigma' = \dim \sigma -1 \}.$$
For a given set of edges $\cA \subseteq E_X$, denote by $G_X^{\cA}$ the graph obtained by reversing the direction
of the edges in $\cA$, i.e., the directed graph with edges\footnote{For the sake of clarity, the arrows that we reverse will be denoted
by $\Longrightarrow$.}
$$E_X^{\cA}= (E_X \setminus \cA) \cup \{ \sigma' \Longrightarrow \sigma \hskip 2mm | \hskip 2mm  \sigma \lra \sigma' \in \cA \}.$$
When each cell of $X$ occurs in at most one edge of $\cA$, we say that $\cA$ is a {\it matching} on $X$.
A matching $\cA$ is {\it acyclic} if the associated graph $G_X^{\cA}$ is acyclic, i.e., does not
contain any directed cycle.
Given an acyclic matching  $\cA$ on $X$, the $\cA$-{\it critical cells} of $X$ are the cells of $X$ that are
not contained in any edge of $\cA$.
Finally, an acyclic matching $\cA$ is {\it homogeneous} whenever $gr(\sigma)=gr(\sigma')$
for any edge $\sigma \lra \sigma' \in \cA$.

\begin{proposition}[{\cite[Proposition 1.2]{BW}}]
Let $(X,gr)$ be a regular $\bZ^n$-graded CW-complex and $\cA$ a homogeneous acyclic matching. Then, there is a
{\rm(}not necessarily regular{\rm)} CW-complex $X_{\cA}$
whose $i$-cells are in one-to-one correspondence with the $\cA$-critical $i$-cells of $X$, such that $X_{\cA}$ is
homotopically equivalent to $X$, and that inherits the $\bZ^n$-graded structure.
\end{proposition}

In the theory of cellular resolutions, we have the following consequence.

\begin{theorem}[{\cite[Theorem 1.3]{BW}}]
Let $I\subseteq R=\kk[x_1,\dots, x_n]$ be a monomial ideal. Assume that $(X,gr)$ is a regular $\bZ^n$-graded
 CW-complex that defines a cellular resolution $\mathbb{F}_{\bullet}^{(X,gr)}$ of $R/I$. Then,  for
a homogeneous acyclic  matching $\cA$ on $G_X$, the $\bZ^n$-graded CW-complex $(X_{\cA},gr)$ supports a cellular
resolution $\mathbb{F}_{\bullet}^{(X_{\cA},gr)}$ of $R/I$.
\end{theorem}

\vskip 2mm

The differentials of the cellular resolution
$\mathbb{F}_{\bullet}^{(X_{\cA},gr)}$ can be explicitely described
in terms of the  differentials of $\mathbb{F}_{\bullet}^{(X,gr)}$
(see \cite[Lemma 7.7]{BW}) as we will see later in Example \ref{ex3gens}.
Since in this paper we are mainly concerned by the Betti numbers, we will not write down the differentials in general.

\subsection{Algebraic discrete Morse theory}

One of the main inconveniences of discrete Morse theory
is that the CW-complex $(X_{\cA},gr)$ that we obtain for a given homogeneous acyclic matching $\cA$
is not necessarily regular. Therefore, we cannot always iterate the procedure.
To overcome such an obstacle, one may use {\it algebraic discrete Morse theory} developed independently by
 Sk\"oldberg \cite{Sko} and J\"ollenbeck and Welker \cite{JW}.

 \vskip 2mm

This approach works directly with an initial free resolution $\mathbb{F}_{\bullet}$
without paying attention whether it has a cellular structure or not.
Given a basis $X=\bigcup_{i\geq 0} X^{(i)}$ of the corresponding free modules $F_i$,
we may consider the directed graph $G_X$ on the set of basis elements with the corresponding
set of edges $E_X$. Then, we may follow the same constructions as in the previous subsection.
Namely, we may define an acyclic matching $\cA \subseteq E_X$ (see \cite[Definition 2.1]{JW})
but, in this case, we have to make sure that the coefficient $[\sigma : \sigma']$ in the differential
corresponding to an edge $\sigma \lra \sigma' \in \cA$ is a unit. We consider the $\cA$-critical
basis elements $X_\cA$ and construct a free resolution $\mathbb{F}_{\bullet}^{X_\cA}$ that is
homotopically equivalent to $\mathbb{F}_{\bullet}$ (see \cite[Theorem 2.2]{JW}).

\section{Pruning the Taylor resolution} \label{prune}

In \cite{BW}, the Lyubeznik resolution is obtained  from the Taylor resolution through discrete Morse theory
by detecting a suitable homogeneous acyclic matching $\cA$ on the simplicial complex $X_{\tt Taylor}$ such that $X_{\tt Lyub}=X_\cA$.
In this section, we use a similar approach to provide some new cellular free resolutions for
monomial ideals. We should point out that the framework considered in
\cite{BW} is slightly more general. To keep notation as simple as possible,
we decided to stick to the case of monomial ideals in a polynomial ring.

\vskip 2mm

\subsection{A cellular free resolution} \label{pruned1}

Let $I=\langle m_1,\dots, m_r \rangle \subseteq R$ be a monomial ideal.
Our starting point is the Taylor resolution $\mathbb{F}_{\bullet}^{(X_{\tt Taylor}, gr)}$.
This resolution is, in general, far from being minimal. In other words, the directed graph $G_{X_{\tt Taylor}}$ associated to ${X_{\tt Taylor}}$
contains a lot of unnecessary information.  Our goal is to prune  this
excess of information in a very simple way. More precisely, we give an algorithm that produces, at each step, a homogeneous
acyclic matching on  ${X_{\tt Taylor}}$. Using
discrete Morse theory, this will provide a cellular free resolution of $R/I$.
It will not be minimal in general, but it will be smaller than the Lyubeznik resolution.

\vskip 2mm

\begin{algorithm}\label{alg1} { (Pruned resolution)}

\vskip 2mm

{\rm \noindent {\sc Input:} The set of edges $E_{X_{\tt Taylor}}$.

\vskip 2mm

For $j$ from $1$ to $r$, incrementing by $1$

\vskip 2mm

\begin{itemize}

\item[\textbf{(j)}] ${\it Prune}$ the edge ${\sigma + \varepsilon_j} \lra{\sigma }$ for
all $\sigma \in \{0,1\}^r$ such that $\sigma_j=0$, where `prune'
means remove the edge\footnote{When we remove an edge, we also remove its two vertices and all the edges
passing through these two vertices.}
if it survived after step $(j-1)$ and
$gr(\sigma)=gr(\sigma + \varepsilon_j) $.

\end{itemize}

\vskip 2mm

\noindent {\sc Return:} The set $\cA_P$ of edges that have been pruned.

}

\end{algorithm}

\begin{example}\label{ex3gens}
When $r=3$, we can visualize the steps of the algorithm over the directed graph as follows:

 {\tiny
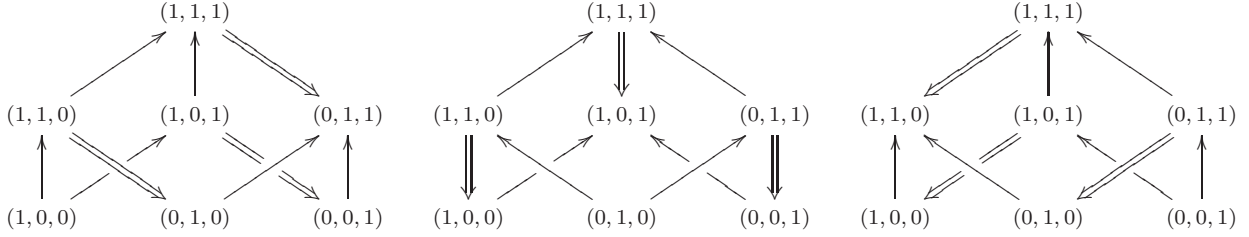
$${\xymatrix { &(1,1,1) \ar[dl] \ar[d] &
\\ (1,1,0) \ar[d]  & (1,0,1) \ar[dl]|\hole & (0,1,1) \ar@2{->}[ul] \ar[d] \ar[dl]
\\ (1,0,0)  &(0,1,0)  \ar@2{->}[ul]    & (0,0,1) \ar@2{->}[ul]| \hole }}
\hskip .41cm
{\xymatrix { &(1,1,1) \ar[dl] \ar[dr] &
\\ (1,1,0) \ar[dr]  & (1,0,1) \ar[dl]|\hole \ar[dr]|\hole \ar@2{->}[u] & (0,1,1) \ar[dl]
\\ (1,0,0) \ar@2{->}[u] &(0,1,0) & (0,0,1) \ar@2{->}[u]  }}
\hskip .41cm
{\xymatrix { &(1,1,1) \ar[dr] \ar[d] &
\\ (1,1,0)  \ar@2{->}[ur]  \ar[d] \ar[dr] & (1,0,1)  \ar[dr]|\hole & (0,1,1)  \ar[d]
\\ (1,0,0) \ar@2{->}[ur]|\hole &(0,1,0)  \ar@2{->}[ur]  &(0,0,1) }}
$$}
The double arrows indicate the direction of the pruning step,
that is, the arrows that will be pruned at each step if the degree of their two vertices coincide
(and if they have not been pruned at a previous step).

Let's illustrate how our algorithm works with an easy example: the edge ideal associated to the 3-cycle,
$I=\langle x_1x_2,x_2x_3,x_1x_3\rangle\subset R=\kk [x_1,x_2,x_3]$. The Taylor resolution of $I$ is
\begin{eqnarray*}
0\rightarrow R
\xrightarrow{\tiny\begin{pmatrix} 1\\ -1\\ 1\end{pmatrix}}
R^3
\xrightarrow{\tiny\begin{pmatrix} x_3&x_3&0\\ -x_1&0&x_1\\ 0&-x_2&-x_2\end{pmatrix}}
R^3
\xrightarrow{\tiny\begin{pmatrix}  x_1x_2 & x_2x_3 & x_1x_3\end{pmatrix}}
I\rightarrow 0&.
\end{eqnarray*}
The directed graph $G_{X_{\tt Taylor}}$ associated to this resolution is 
{\tiny
$$
{\xymatrix { &x_1x_2x_3 \ar[dl] \ar[d] \ar[dr]&
\\ x_1x_2x_3 \ar[d] \ar[dr] & x_1x_2x_3 \ar[dl]|\hole \ar[dr]|\hole & x_1x_2x_3  \ar[d] \ar[dl]
\\ x_1x_2  &x_2x_3 &x_1x_3 }}
$$}where, in order to identify clearly which edges can be pruned, we label the vertices by the monomial $m_\sigma$ instead of $\sigma$.
At the first step of the algorithm  we prune  the arrow $(0,1,1)\Rightarrow (1,1,1)$
as indicated in the first graph below.  Once we remove this edge, we observe on the second graph that no more arrows can be pruned at the next two steps.

{\tiny
$$
{\xymatrix { &x_1x_2x_3 \ar[dl] \ar[d] &
\\ x_1x_2x_3 \ar[d] \ar[dr] & x_1x_2x_3 \ar[dl]|\hole \ar[dr]|\hole & x_1x_2x_3 \ar@2{->}[ul] \ar[d] \ar[dl]
\\ x_1x_2  &x_2x_3 &x_1x_3 }}
\hskip 1.41cm
{\xymatrix { & &
\\ x_1x_2x_3 \ar[d] \ar[dr] & x_1x_2x_3 \ar[dl]|\hole \ar[dr] & 
\\ x_1x_2  &x_2x_3&x_1x_3 }}
$$}

\noindent
The pruned resolution, which is minimal in this case, is as follows:
\begin{eqnarray*}
0\rightarrow 
R^2
\xrightarrow{\tiny\begin{pmatrix} x_3&x_3\\ -x_1&0\\ 0&-x_2\end{pmatrix}}
R^3
\xrightarrow{\tiny\begin{pmatrix}  x_1x_2 & x_2x_3 & x_1x_3\end{pmatrix}}
I\rightarrow 0&.
\end{eqnarray*}
\end{example}

The main result in this section is Theorem \ref{T1} where we show that the output $\cA_P$ of Algorithm \ref{alg1} is
a homogeneous matching such that $(X_{\cA_P},gr)$ supports a cellular free resolution of $R/I$.
In order to do so, we will first prove that each step of Algorithm \ref{alg1}, produces a homogeneous acyclic matching.
In other words, pruning in a fixed direction gives a new smaller free resolution.

\begin{proposition}\label{step}
Let $ E\subseteq E_{X_{\tt Taylor}}$ be a subset of edges and $\cA_j \subseteq E$ be the set of pruned edges in the direction of $\varepsilon_j$. 
Namely, $\cA_j$ are the edges ${\sigma + \varepsilon_j} \lra{\sigma }$ in  $E$ such that $\sigma_j=0$ and  $gr(\sigma)=gr(\sigma + \varepsilon_j) $. 
Then $\cA_j$ is a homogeneous acyclic matching.
\end{proposition}

\begin{proof}
In the graph $G_{X_{\tt Taylor}}$, every vertex $\sigma'$ is involved in exactly one edge of the form ${\sigma + \varepsilon_j} \lra{\sigma }$:
$\sigma'=\sigma$ if $\sigma'_j=0$ and $\sigma'=\sigma+ \varepsilon_j$ otherwise. Thus, $\cA_j$ is a matching.
It is acyclic since we are only reversing arrows in a fixed direction and it is impossible to construct a cycle in $G_{X_{\tt Taylor}}$
with this restriction. It is homogeneous by construction of $\cA_j$ in our pruning step.
\end{proof}

As a direct consequence, we get our desired cellular free resolution.

\begin{theorem} \label{T1} 
Let $I \subseteq R=\kk[x_1, \ldots, x_n]$ be a monomial ideal and
$\cA_P \subseteq E_{X_{\tt Taylor}}$ be the set of pruned edges obtained using Algorithm \ref{alg1}.
Then, the $\bZ^n$-graded CW-complex $(X_{\cA_P},gr)$
supports a cellular free resolution $\mathbb{F}_{\bullet}^{(X_{\cA_P},gr)}$ of $R/I$.
\end{theorem}

\begin{remark}
The free resolution $\mathbb{F}_{\bullet}^{(X_{\cA_P},gr)}$, like the Lyubeznik resolution, strongly depends on the order of the
generators of the monomial ideal $I$.
In general, it is neither simplicial (while the Lyubeznik resolution is always simplicial) nor minimal.
\end{remark}

A nice feature about the pruning algorithm is that
we do not need to care if the original system of generators of the monomial ideal
is minimal or not.
Roughly speaking, the pruning algorithm will always remove the
excess of information given by the extra generators.

\vskip 2mm

\begin{lemma}\label{minimal}
Let $I =\langle m_1,\dots, m_s \rangle \subseteq R$ be a monomial
ideal, and let $\{m_{i_1},\dots, m_{i_r} \}$ be its minimal set of generators
with $1\leq {i_1} < \cdots < {i_r} \leq s$.
Let $X_s$ and $X_r$
be the Taylor simplicial complexes associated to these two sets of
generators\footnote{We have that $X_r$ is a
subcomplex of $X_s$.} and  $\cA_P^s \subseteq E_{X_s}$, $\cA_P^r \subseteq
E_{X_r}$ be  the sets of pruned edges obtained by applying Algorithm
\ref{alg1} to each case. Then, there is an isomorphism of
$\bZ^n$-graded CW-complexes $(X_{\cA_P^r},gr) \cong
(X_{\cA_P^s},gr)$. In particular, both sets of generators lead to
the same cellular free resolution of $R/I$ in Theorem \ref{T1}.
\end{lemma}

\begin{proof}
Assume that $m_j$
is a generator of $I$ that does not belong to the minimal generating set $\{m_{i_1},\dots,
m_{i_r} \}$. We want to check that all the vertices
$\sigma=(\sigma_1,\dots,\sigma_s)$ with $\sigma_j=1$ in the graph
associated to the Taylor complex $X_s$ are pruned using Algorithm
\ref{alg1}.

\vskip 2mm

There exists a minimal generator $m_{i_k}$ such that $m_{i_k} |
m_j$. Therefore, at  step $(i_k)$ of Algorithm \ref{alg1}, we must
prune all the edges $\sigma + \varepsilon_{i_k} \lra \sigma$
with $\sigma_j=1$ and $\sigma_{i_k}=0$
that have survived in the previous steps.  In the case that the edge $\sigma + \varepsilon_\ell
+\varepsilon_{i_k} \lra \sigma  +\varepsilon_{i_k}$ has been pruned
at a previous step $(\ell)$ of the algorithm,  then $gr(\sigma + \varepsilon_\ell
+\varepsilon_{i_k}) = gr(\sigma  +\varepsilon_{i_k}) = gr(\sigma)$ and thus $gr(\sigma)=gr(\sigma + \varepsilon_\ell)$ so
the edge $\sigma + \varepsilon_\ell  \lra
\sigma  $ has to be pruned at the same step $(\ell)$. In particular neither the vertex $\sigma$ nor the vertex $\sigma + \varepsilon_{i_k}$
survive at this previous step $(\ell)$ of the algorithm.
\end{proof}

\subsection{A simplicial free resolution} \label{simplicial1}
If $\cA_P$ is the output of Algorithm \ref{alg1},  the cellular complex $X_{\cA_P}$ may not be simplicial, that is, it may not satisfy the property that
given $\tau \in X_{\cA_P}$, then $\sigma\in X_{\cA_P}$ for any $\sigma \leq \tau$.
In other words,
the pruned free resolution $\mathbb{F}_{\bullet}^{(X_{\cA_P},gr)}$
obtained in Theorem \ref{T1} is always cellular but it may not be simplicial.
If we choose carefully the edges that we prune in  Algorithm \ref{alg1} in order to preserve this property,  we will obtain a
simplicial free resolution of $R/I$ that will be bigger, in general, than the one in Theorem \ref{T1}.

\vskip 2mm

\begin{algorithm}\label{alg3} { (Simplicial pruned resolution)}

\vskip 2mm

{\rm \noindent {\sc Input:} The set of edges $E_{X_{\tt Taylor}}$.

\vskip 2mm

For $j$ from $1$ to $r$, incrementing by $1$

\vskip 2mm

\begin{itemize}

\item[\textbf{(j)}] ${\it Prune}$ the edge ${\sigma} \lra{\sigma + \varepsilon_j}$ for
all $\sigma \in \{0,1\}^r$ such that $\sigma_j=0$ , where `prune'
means remove the edge if it survived after step $(j-1)$,
$gr(\sigma)=gr(\sigma + \varepsilon_j) $ and no
face $\tau > \sigma$ survives at this step $(j)$.

\end{itemize}}

\vskip 2mm

\noindent {\sc Return:} The set $\cA_S$ of edges that have been pruned.

\end{algorithm}

Using the same idea as in Proposition \ref{step}, we have that each step of Algorithm \ref{alg3} produces 
a homogeneous acyclic matching on  $X_{\tt Taylor}$
and hence, the free resolution
$\mathbb{F}_{\bullet}^{(X_{\cA_S},gr)}$ is a simplicial free
resolution of the monomial ideal $I$ if $\cA_S$ is the output of Algorithm \ref{alg3}.

\begin{remark}\label{rkIterate}
It could be the case that the simplicial pruned resolution that we obtain with Algorithm \ref{alg3} may admit further refinement.
This happens when an edge ${\sigma + \varepsilon_j} \lra{\sigma }$,
that would be pruned at step $(j)$ of Algorithm \ref{alg1}, may not be pruned at step $(j)$ of Algorithm \ref{alg3}
because there is a face $\tau > \sigma$ that survives at this step but this face $\tau$ is
pruned in a posterior step.  In this case 
we may apply Algorithm \ref{alg3} once again, with the set of edges $E_{X_{\cA_S}}$ as input instead of $E_{X_{\tt Taylor}}$,
in order to prune the edge ${\sigma + \varepsilon_j } \lra{\sigma }$. Indeed we may repeat Algorithm \ref{alg3} until we have nothing left to prune.
We will observe this phenomenon later in Example \ref{ex5path5cycle}
contructing the simplicial pruned resolution of the 5-cycle.

\end{remark}

\subsection{The Lyubeznik resolution revisited}

The Lyubeznik resolution can be also obtained from the Taylor resolution using our pruning algorithm.
In this case, the edges of the Taylor complex  $X_{\tt Taylor}$ that we prune are obtained
using the following:

\begin{algorithm}\label{alg2} { (The Lyubeznik resolution via pruning)}

\vskip 2mm

{\rm \noindent {\sc Input:} The set of edges $E_{X_{\tt Taylor}}$.

\vskip 2mm

For $j$ from $1$ to $r$, incrementing by $1$

\vskip 2mm

\begin{itemize}

\item[\textbf{(j)}] ${\it Prune}$ the edge ${\sigma} \lra{\sigma + \varepsilon_j}$ for
all $\sigma \in \{0,1\}^r$ such that $\sigma_i=0$ for all $i\leq j$,
where `prune' means remove the edge if it survived after step $(j-1)$
and $gr(\sigma)=gr(\sigma + \varepsilon_j) $.

\end{itemize}}

\vskip 2mm

\noindent {\sc Return:} The set $\cA_L$ of edges that have been pruned.
\end{algorithm}

As in Proposition \ref{step}, each step of Algorithm \ref{alg2} produces a homogeneous acyclic matching on $X_{\tt Taylor}$
and since $X_{\cA_L}=X_{\tt Lyub}$, this provides a simple proof of the result in
\cite[\S3]{BW} that shows that the Lyubeznik resolution can be obtained using discrete Morse theory. 
We also have that the Lyubeznik resolution is
simplicial because the pruning that we consider in Algorithm \ref{alg2}, as the one in Algorithm \ref{alg3},
preserves the simplicial property. In this sense, we may
understand the pruned resolution given in Theorem \ref{T1} and the
simplicial pruned resolution given in Subsection \ref{simplicial1}, as refinements of the Lyubeznik resolution.
Finally, we mention that generalizations of Lyubeznik resolutions have been studied by Novik in \cite{Nov}
using a completely different approach.

\subsection{Some variants of the pruning algorithm} \label{versions}
The methods developed in this work can be extended in several different directions.
The aim of this subsection is to present some of them.

\vskip 2mm

$\cdot$ {\it General setup:} Batzies and Welker \cite{BW} use a slightly more general framework
for discrete Morse theory than the one considered in this paper. The interested reader
should be able to adapt our pruning algorithm to their framework.

\vskip 2mm

$\cdot$ {\it Partial pruning:} Notice that the acyclic matchings considered in Algorithms \ref{alg1}, \ref{alg3}
and \ref{alg2} satisfy $\cA_P \supseteq \cA_S \supseteq \cA_L.$ One may also consider any convenient subset
$\cA_P \supseteq \cA'$ of pruning edges. Indeed, we may consider
many different variants using algebraic discrete Morse theory iteratively.
We only have to  pick a convenient edge at each iteration.

\vskip 2mm

$\cdot$ {\it Pruning other resolutions:} The method that we present here always starts with the Taylor resolution but
 we may start with other non-minimal free resolutions, the Lyubeznik resolution for example.
The advantage of the Taylor resolution is that it does not depend on the order of the generators, and
the simplicity of its construction makes our algorithm very easy to present and implement.

\vskip 2mm

$\cdot${\it $\nu$-invariants:}  A new set of invariants that measure the acyclicity of the linear strands of a minimal free
resolution of a graded ideal was introduced in \cite{AY}.  We may obtain an approximation to these invariants by applying
the following iteration of the pruning algorithm. We first apply Algorithm \ref{alg1} to obtain the
CW-complex $X_{\cA_P}$ and its corresponding free resolution. Then, we apply  the pruning algorithm to
$X_{\cA_P}$ with the following variant:

\vskip 2mm

\begin{itemize}

\item[\textbf{(j)}] ${\it Prune}$ the edge ${\sigma + \varepsilon_j} \lra{\sigma }$ for
all $\sigma \in \{0,1\}^r$ such that $\sigma_j=0$, where `prune'
means remove the edge if it survived after step $(j-1)$ and
$gr(\sigma)=gr(\sigma + \varepsilon_j) -1 $.
Here $gr(\sigma)=|\alpha| \in \bZ$ where ${\bf x}^\alpha=m_\sigma$.

\end{itemize}

\section{Examples} \label{examples}

In this section, we will illustrate that the
pruned free resolution described in Section \ref{pruned1} is fairly
close to the minimal free resolution in some examples.
Indeed, we will compare the pruned resolution to the
simplicial free resolution obtained in Section \ref{simplicial1} and to the
Lyubeznik resolution by means of their corresponding Betti diagram.

\vskip 2mm

\subsection{First examples}

Already for some simple examples we can appreciate the better behavior of the pruned
resolutions with respect to the Lyubeznik resolution.

\begin{example}\label{ex5path5cycle}
We are going to describe the steps of the pruning Algorithms described in Section \ref{prune}
for the edge ideals associated to a $5$-path and a $5$-cycle:

\vskip 2mm

$\bullet$ Let  $I=\langle {x}_{1} {x}_{2}, {x}_{2} {x}_{3}, {x}_{3} {x}_{4}, {x}_{4} {x}_{5}\rangle$ be the edge ideal
of a $5$-path. The steps performed with Algorithm \ref{alg1} are the following:

\vskip 2mm

\noindent $\cdot$ {\bf Step (1):} No edge is pruned.

\noindent $\cdot$ {\bf Step (2):} We prune the edges $(1,0,1,0) \Rightarrow (1,1,1,0)$ and $(1,0,1,1) \Rightarrow (1,1,1,1)$.

\noindent $\cdot$ {\bf Step (3):} We prune the edge $(0,1,0,1) \Rightarrow (0,1,1,1)$.

\noindent $\cdot$ {\bf Step (4):} No edge is pruned.

\vskip 2mm

The two edges pruned in Step (2) are also pruned using Algorithm \ref{alg3} but are
not pruned using Algorithm \ref{alg2}. The edge pruned in Step (3) is not pruned neither
using Algorithm \ref{alg3} nor Algorithm \ref{alg2}. Therefore, the Betti diagrams of the pruned, the simplicial pruned,
and the Lyubeznik resolutions
are  respectively:

\vskip 2mm

\begin{center}
{\small $\begin{matrix}
&0&1&2&3\\\text{total:}&1&4&4&1\\\text{0:}&1&\text{.}&\text{.}&\text{.}\\\text{1:}&\text{.}&4&3&\text{.}\\\text{2:}&\text{.}&\text{.}&1&1\\\end{matrix}
$
\hfill
$\begin{matrix}
&0&1&2&3\\
\text{total:}&1&4&5&2\\
\text{0:}&1&\text{.}&\text{.}&\text{.}\\
\text{1:}&\text{.}&4&3&1\\
\text{2:}&\text{.}&\text{.}&2&1\\
\end{matrix}$
\hfill
$\begin{matrix}
&0&1&2&3&4\\
\text{total:}&1&4&6&4&1\\
\text{0:}&1&\text{.}&\text{.}&\text{.}&\text{.}\\
\text{1:}&\text{.}&4&3&2&1\\
\text{2:}&\text{.}&\text{.}&3&2&\text{.}\\
\end{matrix}$
}
\hskip 1cm
\end{center}
The pruned resolution is minimal. The simplicial pruned resolution is not minimal but it is as small
as possible since the ideal $I$ has no minimal simplicial resolution as one can check using an argument
similar to the one used in \cite[Section 3.2]{Jac04} for the 4-cycle. The Lyubeznik resolution coincides with the Taylor
resolution since no edge has been pruned through Algorithm \ref{alg2}.

\vskip 2mm

$\bullet$  Let $I=\langle {x}_{1} {x}_{2}, {x}_{2} {x}_{3}, {x}_{3} {x}_{4}, {x}_{4} {x}_{5}, {x}_{5} {x}_{1}\rangle$
be the edge ideal of a $5$-cycle. The steps performed with Algorithm \ref{alg1} are the following:

\vskip 2mm

\noindent $\cdot$ {\bf Step (1):} We prune four edges $(0,1,0,0,1) \Rightarrow (1,1,0,0,1)$,
$(0,1,1,0,1) \Rightarrow (1,1,1,0,1)$, $(0,1,0,1,1) \Rightarrow (1,1,0,1,1)$ and $(0,1,1,1,1) \Rightarrow (1,1,1,1,1)$.

\noindent $\cdot$ {\bf Step (2):} We prune  $(1,0,1,0,0) \Rightarrow (1,1,1,0,0)$ and $(1,0,1,1,0) \Rightarrow (1,1,1,1,0)$.

\noindent $\cdot$ {\bf Step (3):} We prune the edge $(0,1,0,1,0) \Rightarrow (0,1,1,1,0)$.

\noindent $\cdot$ {\bf Step (4):}  We prune  $(0,0,1,0,1) \Rightarrow (0,0,1,1,1)$ and $(1,0,1,0,1) \Rightarrow (1,0,1,1,1)$.

\noindent $\cdot$ {\bf Step (5):} We prune the edge $(1,0,0,1,0) \Rightarrow (1,0,0,1,1)$.

\vskip 2mm

The edges pruned at Step (1) are also pruned using Algorithm \ref{alg3} and Algorithm \ref{alg2}.
The two edges pruned at Step (2) can not be pruned  using Algorithm \ref{alg3} because the vertices
$(1,0,1,0,1)$ and $(1,0,1,1,1)$ remain at this step of the algorithm. The edges pruned at Step (3)
and (4) are also pruned in Algorithm \ref{alg3}. The edge pruned at Step (5) is not pruned in Algorithm \ref{alg3}.
Notice that the two edges that prohibited the pruning at Step (3) of Algorithm \ref{alg3} were pruned at Step (4).
This means that we can perform another round of Algorithm \ref{alg3} that will prune the two edges that could not be pruned
before at step (2) as observed in Remark \ref{rkIterate}.

\vskip 2mm

On the other hand, note that after the first step, no more edge can be pruned through Algorithm \ref{alg2}.
We thus obtain that the Betti diagrams of the pruned, the simplicial pruned, and the Lyubeznik resolutions are
the following respectively:

\vskip 2mm

\begin{center}
{\small
$\begin{matrix}
&0&1&2&3\\\text{total:}&1&5&5& 1\\\text{0:}&1&\text{.}&\text{.}&\text{.}\\\text{1:}&\text{.}&5&5&\text{.}\\\text{2:}&\text{.}&\text{.}&\text{.}&1\\\end{matrix}
$
\hfill
$\begin{matrix}
&0&1&2&3\\
\text{total:}&1&5&6&2\\
\text{0:}&1&\text{.}&\text{.}&\text{.}\\
\text{1:}&\text{.}&5&5&1\\
\text{2:}&\text{.}&\text{.}&1&1\\
\end{matrix}$
\hfill
$\begin{matrix}
&0&1&2&3&4\\
\text{total:}&1&5&9&7&2\\
\text{0:}&1&\text{.}&\text{.}&\text{.}&\text{.}\\
\text{1:}&\text{.}&5&5&4&2\\
\text{2:}&\text{.}&\text{.}&4&3&\text{.}\\
\end{matrix}$
}
\hskip 1cm
\end{center}
As in the previous example, the pruned resolution is minimal and the simplicial pruned resolution is not minimal but it is as small as possible.
\end{example}

\begin{remark}
It is not surprising that in the previous examples we could find a cellular minimal resolution since both
 the 5-path and the 5-cycle fall into \cite[Example 1.7]{BS}.
\end{remark}

\vskip 2mm

For larger examples the difference  between the pruned and the Luybeznik resolutions
becomes more apparent.

\vskip 2mm

\begin{example}
Consider the ideal
$
I=\langle {x}_{1}^{4}, {x}_{2}^{4}, {x}_{2}^{2} {x}_{3}^{2}, {x}_{3}^{4}, {x}_{4}^{4}, {x}_{1} {x}_{4}^{2} {x}_{5},
{x}_{5}^{4}, {x}_{2}^{2} {x}_{6}^{2}, {x}_{6}^{4}, {x}_{4}^{2} {x}_{7}^{2}, {x}_{7}^{4}\rangle.
$
The pruned resolution is minimal while the Lyubeznik resolution is not. The Betti diagrams are:

\vskip 2mm

\begin{center}
{\tiny $\begin{matrix}
&0&1&2&3&4&5&6&7\\\text{total:}&1&11&49&114&148&107&40&6\\\text{0:}&1&\text{.}&\text{.}&\text{.}&\text{.}&\text{.}&\text{.}&\text{.}\\\text{1:}&\text{.}&\text{.}&\text{.}&\text{.}&\text{.}&\text{.}&\text{.}&\text{.}\\\text{2:}&\text{.}&\text{.}&\text{.}&\text{.}&\text{.}&\text{.}&\text{.}&\text{.}\\\text{3:}&\text{.}&11&\text{.}&\text{.}&\text{.}&\text{.}&\text{.}&\text{.}\\\text{4:}&\text{.}&\text{.}&9&\text{.}&\text{.}&\text{.}&\text{.}&\text{.}\\\text{5:}&\text{.}&\text{.}&2&2&\text{.}&\text{.}&\text{.}&\text{.}\\\text{6:}&\text{.}&\text{.}&38&4&\text{.}&\text{.}&\text{.}&\text{.}\\\text{7:}&\text{.}&\text{.}&\text{.}&56&2&\text{.}&\text{.}&\text{.}\\\text{8:}&\text{.}&\text{.}&\text{.}&10&31&\text{.}&\text{.}&\text{.}\\\text{9:}&\text{.}&\text{.}&\text{.}&42&30&9&\text{.}&\text{.}\\\text{10:}&\text{.}&\text{.}&\text{.}&\text{.}&71&32&1&\text{.}\\\text{11:}&\text{.}&\text{.}&\text{.}&\text{.}&2&37&14&\text{.}\\\text{12:}&\text{.}&\text{.}&\text{.}&\text{.}&12&6&6&2\\\text{13:}&\text{.}&\text{.}&\text{.}&
\text{.}&\text{.}&22&6&\text{.}\\\text{14:}&\text{.}&\text{.}&\text{.}&\text{.}&\text{.}&\text{.}&11&2\\\text{15:}&\text{.}&\text{.}&\text{.}&\text{.}&\text{.}&1&\text{.}&1\\\text{16:}&\text{.}&\text{.}&\text{.}&\text{.}&\text{.}&\text{.}&2&\text{.}\\\text{17:}&\text{.}&\text{.}&\text{.}&\text{.}&\text{.}&\text{.}&\text{.}&1\\\end{matrix}
$ \hskip 1cm   $\begin{matrix}
&0&1&2&3&4&5&6&7&8&9&10\\\text{total:}&1&11&54&156&294&378&336&204&81&19&2\\\text{0:}&1&\text{.}&\text{.}&\text{.}&\text{.}&\text{.}&\text{.}&\text{.}&\text{.}&\text{.}&\text{.}\\\text{1:}&\text{.}&\text{.}&\text{.}&\text{.}&\text{.}&\text{.}&\text{.}&\text{.}&\text{.}&\text{.}&\text{.}\\\text{2:}&\text{.}&\text{.}&\text{.}&\text{.}&\text{.}&\text{.}&\text{.}&\text{.}&\text{.}&\text{.}&\text{.}\\\text{3:}&\text{.}&11&\text{.}&\text{.}&\text{.}&\text{.}&\text{.}&\text{.}&\text{.}&\text{.}&\text{.}\\\text{4:}&\text{.}&\text{.}&9&\text{.}&\text{.}&\text{.}&\text{.}&\text{.}&\text{.}&\text{.}&\text{.}\\\text{5:}&\text{.}&\text{.}&2&7&\text{.}&\text{.}&\text{.}&\text{.}&\text{.}&\text{.}&\text{.}\\\text{6:}&\text{.}&\text{.}&43&4&2&\text{.}&\text{.}&\text{.}&\text{.}&\text{.}&\text{.}\\\text{7:}&\text{.}&\text{.}&\text{.}&58&4&\text{.}&\text{.}&\text{.}&\text{.}&\text{.}&\text{.}\\\text{8:}&\text{.}&\text{.}&\text{.}&12&64&2&\text{.}&\text{.}&\text{.}&\text{.}&\text{.}\\\text{9:}&\text{.}&\text{.}&\text{.}&75&32&
44&\text{.}&\text{.}&\text{.}&\text{.}&\text{.}\\\text{10:}&\text{.}&\text{.}&\text{.}&\text{.}&106&48&22&\text{.}&\text{.}&\text{.}&\text{.}\\\text{11:}&\text{.}&\text{.}&\text{.}&\text{.}&18&114&48&7&\text{.}&\text{.}&\text{.}\\\text{12:}&\text{.}&\text{.}&\text{.}&\text{.}&68&40&77&30&1&\text{.}&\text{.}\\\text{13:}&\text{.}&\text{.}&\text{.}&\text{.}&\text{.}&86&44&42&12&\text{.}&\text{.}\\\text{14:}&\text{.}&\text{.}&\text{.}&\text{.}&\text{.}&10&85&30&18&2&\text{.}\\\text{15:}&\text{.}&\text{.}&\text{.}&\text{.}&\text{.}&34&16&50&10&6&\text{.}\\\text{16:}&\text{.}&\text{.}&\text{.}&\text{.}&\text{.}&\text{.}&33&10&23&2&1\\\text{17:}&\text{.}&\text{.}&\text{.}&\text{.}&\text{.}&\text{.}&2&27&4&6&\text{.}\\\text{18:}&\text{.}&\text{.}&\text{.}&\text{.}&\text{.}&\text{.}&9&2&10&\text{.}&1\\\text{19:}&\text{.}&\text{.}&\text{.}&\text{.}&\text{.}&\text{.}&\text{.}&5&\text{.}&3&\text{.}\\\text{20:}&\text{.}&\text{.}&\text{.}&\text{.}&\text{.}&\text{.}&\text{.}&\text{.}&3&\text{.}&\text{.}\\\text{21:}&\text{.}
&\text{.}&\text{.}&\text{.}&\text{.}&\text{.}&\text{.}&1&\text{.}&\text{.}&\text{.}\\\end{matrix}
$}
\end{center}
\end{example}

\vskip 2mm

\subsection{Independence on the characteristic of the base field}
It is well-known that the Betti numbers of a monomial ideal depend on the characteristic of the base field.
This phenomenon can be understood through Hochster's formula.
Namely, given  a monomial ideal $I\subseteq R=\kk[x_1,\dots,x_n]$,
we may assume that it is squarefree since its Betti numbers can always be obtained from a squarefree monomial ideal using polarization.
Now, a squarefree monomial ideal $I$ is always the Stanley-Reisner ideal of a simplicial complex $\Delta$.
For any squarefree degree $\alpha \in \{0,1\}^n$, we consider
$F_\alpha:=\{x_i \hskip 2mm | \hskip 2mm \alpha_i \neq 0 \} \subseteq \{x_1,\dots,x_n\}$.
Let $\Delta_{|F_\alpha}$ be the simplicial complex obtained by taking all the faces of $\Delta$
whose vertices belong to $F_\alpha$. Hochster's formula states that the Betti numbers of $R/I$ can be computed
as the dimensions of reduced simplicial homology groups:
$$\beta_{i,\alpha}(R/I)=\dm_{\kk} \widetilde{H}_{|\alpha|-i-1}(\Delta_{|F_\alpha}; \kk).$$
By the universal coefficient theorem, we may just compute the homology modules
over the integer numbers $\widetilde{H}_{|\alpha|-i-1}(\Delta_{|F_\alpha}; \bZ)$. It follows that Betti
numbers depend on the characteristic of the base field whenever these homology groups have torsion.

\vskip 2mm

\begin{example} \label{ex_p}
The most recurrent example that illustrates the behavior of Betti numbers with respect to the characteristic of the base field
is the Stanley-Reisner ideal associated to a minimal triangulation of $\mathbb{P}_{\bR}^2$.
Namely, consider the following ideal in $R=\kk[x_1,\dots, x_6]$:
$$
I=\langle x_1x_2x_3,x_1x_2x_4,x_1x_3x_5,x_2x_4x_5,x_3x_4x_5,x_2x_3x_6,x_1x_4x_6,x_3x_4x_6,x_1x_5x_6,x_2x_5x_6\rangle.
$$
The Betti diagrams in characteristic zero and two respectively are
\begin{center}
{\small $\begin{matrix}
&0&1&2&3\\\text{total:}&1&10&15&6\\\text{0:}&1&\text{.}&\text{.}&\text{.}\\\text{1:}&\text{.}&\text{.}&\text{.}&\text{.}\\\text{2:}&\text{.}&10&15&6\\\end{matrix}
$ \hskip 3cm $\begin{matrix}
&0&1&2&3&4\\\text{total:}&1&10&15&7&1\\\text{0:}&1&\text{.}&\text{.}&\text{.}&\text{.}\\\text{1:}&\text{.}&\text{.}&\text{.}&\text{.}&\text{.}\\\text{2:}&\text{.}&10&15&6&1\\\text{3:}&\text{.}&\text{.}&\text{.}&1&\text{.}\\\end{matrix}
$}
\end{center}
On the other hand, the results that we obtain with the pruned and the Lyubeznik resolution respectively are:

\vskip 2mm

\begin{center}
{\small $\begin{matrix}
&0&1&2&3&4\\\text{total:}&1&10&15&7&1\\\text{0:}&1&\text{.}&\text{.}&\text{.}&\text{.}\\\text{1:}&\text{.}&\text{.}&\text{.}&\text{.}&\text{.}\\\text{2:}&\text{.}&10&15&6&1\\\text{3:}&\text{.}&\text{.}&\text{.}&1&\text{.}\\\end{matrix}
$ \hskip 3cm
$\begin{matrix}
&0&1&2&3&4\\\text{total:}&1&10&27&27&9\\\text{0:}&1&\text{.}&\text{.}&\text{.}&\text{.}\\\text{1:}&\text{.}&\text{.}&\text{.}&\text{.}&\text{.}\\\text{2:}&\text{.}&10&15&18&9\\\text{3:}&\text{.}&\text{.}&12&9&\text{.}\\\end{matrix}
$}
\end{center}
\end{example}

\begin{remark}
All the free resolutions that we consider in this work are obtained
from the Taylor resolution using discrete Morse theory as in
\cite{BW} or \cite{JW}. If one follows closely the construction of the CW-complex
$X_{\cA}$ that is homotopically equivalent to $X_{\tt Taylor}$, one
may check that it is independent of the characteristic of the base
field.
In this sense, the result obtained with the pruned resolution in Example \ref{ex_p} is the best that we can achieve
using discrete Morse theory.
\end{remark}

\subsection{Minimality}

A necessary and sufficient condition for the minimality of a cellular free resolution
$\mathbb{F}_{\bullet}^{(X_{\cA},gr)}$  is described in \cite[Lemma 7.5]{BW}.
The pruned resolution obtained in Theorem \ref{T1} is minimal
if and only if the following condition is satisfied:
the cells $\sigma\in X_{\cA}$ that
survive to any of the pruning processes described in Section $\S 3$ satisfy
$gr(\sigma) \neq gr(\sigma + \varepsilon_j) $ whenever $\sigma_j=0$.

\vskip 2mm

\begin{remark}
Barile \cite{Bar} noticed that, in the case of the Lyubeznik resolution, it is enough
checking out the aforementioned condition for maximal faces of $X_{\tt Lyub}$.
\end{remark}

\vskip 2mm

The Taylor resolution is rarely minimal.
Some examples of minimal Lyubeznik resolutions
include shellable ideals (see \cite{BW}) and the matroid ideal of a
finite projective space (see \cite{Nov}). Finally, let's point out that Torrente and Varbaro \cite{TV15} provided a
fast algorithm for computing Betti diagrams of a monomial ideal taking as a starting point
the Lyubeznik resolution of the ideal. Using the pruned resolution as starting
point of their algorithm should speed up their computation of the Betti diagram.

\section{Betti splittable monomial ideals} \label{split}
A technique that has been successfully applied to describe Betti
numbers of monomial ideals is based on the concept of splitting
introduced by Eliahou and Kervaire in \cite{EK}; see \cite{HVT2} for a nice survey.
A refinement of this concept was coined by Francisco, H\`a and Van Tuyl
in \cite{FHV}.

\begin{definition}
We say that $I=J+K$ is a {\it Betti splitting} for the monomial ideal $I$ if
the following formula for the $\bZ^n$-graded Betti numbers is satisfied:
 $$\beta_{i,\alpha}(I)= \beta_{i,\alpha}(J)+\beta_{i,\alpha}(K)+\beta_{i-1,\alpha}(J\cap
K)\,.$$
\end{definition}

We can mimic the same concept introducing the {\it pruned Betti numbers} as
the Betti numbers that we obtain in the pruned free resolution of
Section \ref{pruned1}, and  that we will  denote by
$\overline{\beta}_{i,\alpha}(I)$ to distinguish them from the usual Betti
numbers.

\begin{definition}
We say that $I=J+K$ is a {\it pruned Betti splitting} for the monomial ideal $I$ if
the following formula for the $\bZ^n$-graded Betti numbers is satisfied:
 $$\overline{\beta}_{i,\alpha}(I)= \overline{\beta}_{i,\alpha}(J)+\overline{\beta}_{i,\alpha}(K)+\overline{\beta}_{i-1,\alpha}(J\cap
K)\,.$$
\end{definition}

\vskip 2mm

\begin{remark}
Obviously, a pruned Betti splitting provides a Betti splitting
whenever the pruning algorithm gives a minimal free resolution
of $I$.
\end{remark}

\vskip 2mm

Necessary and sufficient conditions describing Betti splittings
have been given in \cite{FHV}. Using our main Algorithm \ref{alg1},
we will give some conditions that are easy to check and that provide a pruned Betti
splitting.

\subsection{A partial pruning of $J\cap K$}

Let $I=\langle m_1,\dots,m_{s+r} \rangle$ be a monomial ideal, and consider the Taylor
simplicial complex $X_{\tt Taylor}$ on $r$ vertices which
faces are labeled by $\sigma \in \{0,1\}^{s+r}$. Consider a
decomposition $I=J+K$ with $J=\langle m_1,\dots,m_s \rangle$ and $K=
\langle m_{s+1},\dots,m_{s+r} \rangle$. In order to compute the pruned
resolution of $J$ and $K$ we have to consider the subcomplexes:

\vskip 2mm

\begin{itemize}
 \item [$\cdot$] $X_J\subseteq X_{\tt Taylor}$  with faces labelled by
 $\sigma=(\sigma_1,\dots,\sigma_s,0,\dots,0) \in \{0,1\}^{s+r}$, and
 \item [$\cdot$] $X_K\subseteq X_{\tt Taylor}$  with faces labelled by
 $\sigma=(0,\dots,0,\sigma_{s+1},\dots,\sigma_{s+r}) \in \{0,1\}^{s+r}$.
\end{itemize}

\vskip 2mm

Denote by $X'$ the set obtained by removing the faces of
$X_J$ and $X_K$ from $X_{\tt Taylor}$. Notice that  $X'$
 is not a simplicial subcomplex of $X_{\tt Taylor}$ and, in particular,
 it is not the Taylor simplicial complex associated to the intersection.
 However, applying a partial pruning algorithm to the Taylor complex  $X_{J\cap K}$
 associated to a (possibly non-minimal) set of generators of  $J\cap K$
  we obtain  $X'$.
 The proof of this fact is quite tedious but straightforward.
 We will simply sketch it  and leave the details for the interested reader.

 \vskip 2mm

  Consider the (possibly non-minimal) set of generators of  $J\cap K$  given by
 the monomials
 $$\{ m_{1,s+1},\dots, m_{s,s+1},  m_{1,s+2}, \dots, m_{s,s+2},  \dots, m_{1,r}, \dots, m_{s,s+r}\},$$
 where $m_{i,s+1+k}=\lcm(m_i,m_{s+1+k})$.  In what follows we will also denote $m_{i_1,\dots, i_\ell}=\lcm(m_{i_1},\dots, m_{i_\ell})$.
The Taylor complex $X_{J\cap K}$ is the full simplicial complex on $sr$ vertices with faces labelled by
$\sigma \in \{0,1\}^{sr}$. Notice that the  vertex $\varepsilon_{j}$  corresponds to the monomial
$m_{\varepsilon_j}:=m_{i,s+1+k}$ if we have
$j=ks+i$ for $k\in \{0,\dots, r-1\}$ and $i\in\{1,\dots, s\}$.  In general, a face $\sigma$ corresponds
to $m_\sigma:=\lcm (m_{\varepsilon_j} \hskip 2mm | \hskip 2mm \sigma_j = 1)$.

 \vskip 2mm

 Now we want to apply our Algorithm \ref{alg1} but we only want to prune the edges
 ${\sigma + \varepsilon_j} \lra{\sigma }$ whenever the corresponding $\lcm$'s involve the same
 monomials $m_i \in I$. For example, we have  that
 $\lcm(m_{i,b}, m_{a,s+1+k})=\lcm(m_{i,s+1+k},m_{i,b}, m_{a,s+1+k})=m_{i,s+1+k,a,b}$  so both $\lcm$'s
 involve the same monomials $m_i,m_{s+1+k},m_a,m_b \in I$ for all $a,b$.
 This means that the corresponding edge would be pruned at step $j=ks+i$.

 \vskip 2mm

 The partial pruning algorithm that we apply is the following:

 \vskip 2mm

 \begin{algorithm}\label{alg4} { (Partial pruning)}

\vskip 2mm

{\rm \noindent {\sc Input:} The set of edges $E_{X_{J\cap K}}$.

\vskip 2mm

For $j$ from $1$ to $sr$, incrementing by $1$

\vskip 2mm

 \begin{itemize}

\item[\textbf{(j)}] ${\it Prune}$ the edge ${\sigma + \varepsilon_j} \lra{\sigma }$ for
all $\sigma \in \{0,1\}^{sr}$ such that $\sigma_j=0$, where `prune'
means remove the edge if  it survived after step $(j-1)$ and
$m_\sigma$ involves the monomials $m_i$ and $m_{s+1+k}$, where $j=ks+i$.

\end{itemize}

\vskip 2mm

\noindent {\sc Return:} The set $\cA'$ of edges that have been pruned.

}

\end{algorithm}

The tedious part is to check that using this algorithm we obtain $X'$, that is, $X_{\cA'}=X'$.
To illustrate this fact we present the following:

\begin{example}
Let $I=\langle m_1,m_2,m_3,m_4 \rangle$ be a monomial ideal generated by four monomials.
We are going to consider the following decompositions:

\vskip 2mm

$\cdot$ Consider  $I=J+K$ with $J=\langle m_1,m_2,m_3 \rangle$ and $K= \langle m_4 \rangle$.
In this case we have  $J\cap K= \langle m_{1,4}, m_{2,4},m_{3,4} \rangle $  so  $X'$
coincides with $X_{J\cap K}$.

\vskip 2mm

$\cdot$ If  $J=\langle m_1,m_2 \rangle$ and $K= \langle m_3, m_4 \rangle$ then
we have that $J\cap K= \langle m_{1,3}, m_{2,3}, m_{1,4}, m_{2,4} \rangle $.
The directed graph associated to the Taylor simplicial complex of $I$ is:

{\tiny 
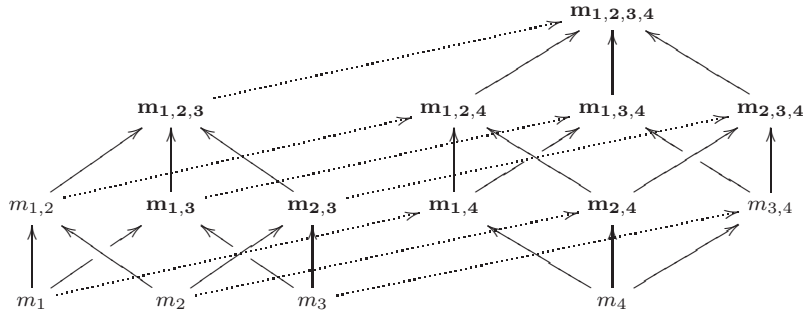
$${\xymatrix { &&&& {\bf m_{1,2,3,4}} \ar@{.>}[dlll] \ar[d] \ar[dl] \ar[dr] & \\
& {\bf m_{1,2,3}}  \ar[d] \ar[dl] \ar[dr] & & {\bf m_{1,2,4}} \ar@{.>}[dlll] \ar[d] \ar[dr] & {\bf m_{1,3,4}} \ar@{.>}[dlll] \ar[dr]|\hole \ar[dl]|\hole& {\bf m_{2,3,4}} \ar[dl] \ar[d]\\
m_{1,2}  \ar[d] \ar[dr] &{\bf m_{1,3}} \ar[dr]|\hole \ar[dl]|\hole & {\bf m_{2,3}}\ar[dl] \ar[d]& {\bf m_{1,4}} \ar@{.>}[dlll] \ar[dr] & {\bf m_{2,4}}  \ar@{.>}[dlll]\ar[d]& m_{3,4}  \ar@{.>}[dlll] \ar[dl]\\
m_{1} &m_{2} & m_{3} &  & m_{4}  &  \\
}}
$$}

\noindent
and
the monomials corresponding to the set $X'$ are indicated with bold letters.
On the other hand, applying the partial pruning algorithm to the Taylor complex $X_{J\cap K}$ we also get
the set $X'$ as shown in the following graph:

{\tiny 
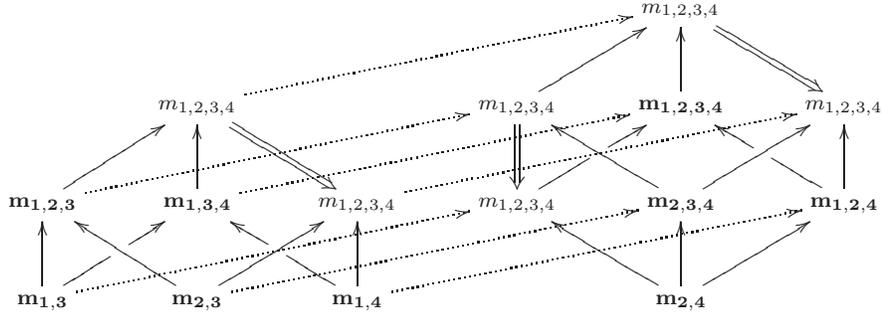
$${\xymatrix { &&&& m_{1,2,3,4} \ar@{.>}[dlll] \ar[d] \ar[dl]& \\
& m_{1,2,3,4} \ar[d] \ar[dl]  & & m_{1,2,3,4}  \ar@{.>}[dlll]  \ar[dr] & {\bf m_{1,2,3,4}} \ar@{.>}[dlll] \ar[dr]|\hole \ar[dl]|\hole & m_{1,2,3,4} \ar[dl] \ar[d] \ar@2{->}[ul] \\
{\bf m_{1,2,3}} \ar[d] \ar[dr] & {\bf m_{1,3,4}} \ar[dr]|\hole \ar[dl]|\hole & m_{1,2,3,4}\ar[dl] \ar[d]  \ar@2{->}[ul]& m_{1,2,3,4} \ar@2{->}[u] \ar@{.>}[dlll] \ar[dr] & {\bf m_{2,3,4}} \ar@{.>}[dlll]\ar[d]& {\bf m_{1,2,4}} \ar@{.>}[dlll] \ar[dl]\\
{\bf m_{1,3}} & {\bf m_{2,3}} & {\bf m_{1,4}} &  & {\bf m_{2,4}}  &  \\
}}
$$}

\end{example}

\subsection{A sufficient condition for Betti splittings}

Let $I=J+K$ be a decomposition of a monomial ideal.
In order to have a Betti splitting, it is enough to check that applying Algorithm \ref{alg1} to $I$,
the pruning steps are realized
independently within the edges associated to $X_J$, $X_K$ and $X'$. Notice that in this case we obtain the pruned Betti numbers of $I$ from the pruned Betti numbers of $J$, $K$ and $J\cap K$ and thus we get a sufficient condition for a Betti splitting. We collect this fact in the following:

\begin{proposition}
Let $I=\langle m_1,\dots,m_{s+r} \rangle\subseteq \kk[x_1,\dots,x_n]$ be a monomial
ideal and consider the ideals $J=\langle m_1,\dots,m_s \rangle$ and $K=
\langle m_{s+1},\dots,m_{s+r} \rangle$.  If  we only perform pruning steps within the
edges of $X_J$, $X_K$ or $X'$ separately in Algorithm
\ref{alg1}, then $I=J+K$ is a pruned Betti splitting for $I$.
\end{proposition}

In the particular case that we remove just one generator from our initial ideal,
we obtain the following:

\begin{corollary}
Let $I=\langle m_1,\dots,m_{s+1} \rangle\subseteq \kk[x_1,\dots,x_n]$ be
a monomial ideal and consider the ideals $J=\langle
m_1,\dots,m_{s} \rangle$ and $K= \langle m_{s+1} \rangle$. If no
pruning has been made at the step $s+1$ of Algorithm \ref{alg1}, then
$I=J+K$ is a pruned Betti splitting for $I$.
\end{corollary}

\vskip 2mm

 Betti splitting techniques can be used to provide a recursive
 method to compute Betti numbers of monomial ideals. For example,
 the case of edge ideals of graphs has been successfully studied in \cite{HVT1} and \cite{FHV}
where they considered {\it splitting edges} and {\it
splitting vertices}. Splitting edges are easy to describe (see
\cite[Theorem 3.4]{FHV}) but, in general they are difficult to find.
On the other hand, every vertex is a splitting vertex except for some limit cases where the vertex is
isolated or its complement consists of isolated vertices.

\vskip 2mm
One can also check that any vertex also
provides a pruned Betti splitting.
Indeed, let $I=I(G)\subseteq \kk[x_1,\dots,x_n]$ be the edge ideal of a
graph $G$ on $n$ vertices. Consider a vertex, say $x_n$, and the
decomposition $I=J+K$ where $J=I(G\backslash \{x_n\})$ is the edge
ideal of the graph obtained from $G$ by removing
the vertex $x_n$ and all the edges passing through this vertex. Notice that  $K$ is the edge ideal of a bipartite graph
$\mathcal{K}_{1,d}$ where $d$ is the number of vertices adjacent to $x_n$ in $G$. We
have a pruned Betti splitting because the ideal $J$ does not involve
the variable $x_n$ and the pruning steps of the algorithm are
performed within  $X_{K}$ and $X'$ separately. Moreover, it
is not difficult to see that the pruning algorithm provides a minimal
free resolution for $K$, and hence
$$\overline{\beta}_{i,\alpha}(I)= \overline{\beta}_{i,\alpha}(J)+{\beta}_{i,\alpha}(K)+\overline{\beta}_{i-1,\alpha}(J\cap
K).$$
In particular, if Algorithm \ref{alg1} provides a minimal free
resolution for $J$ and $J\cap K$, then it is also provides a minimal free resolution for $I$.

\vskip 2mm

For paths and cycles, we can use these splitting
techniques to prove that the pruning algorithm will always provide a minimal
free resolution as we already observed in Example \ref{ex5path5cycle} for the $5$-path and the $5$-cycle.
The Betti numbers of these ideals have already
been computed by Jacques in \cite{Jac04}.

\vskip 2mm

\begin{example}[$n$-paths]\label{ex_npath}
Let $I=\langle
x_1x_2,x_2x_3,\dots,x_{n-1}x_n \rangle$ be the edge ideal of an
$n$-path. A decomposition $I=J+K$ with $J=\langle
x_1x_2,x_2x_3,\dots,x_{n-2}x_{n-1} \rangle$ and $K=\langle
x_{n-1}x_n \rangle$ is a pruned Betti splitting simply because the
ideal $J$ does not involve the vertex $x_n$. Therefore we have
$$\overline{\beta}_{i,\alpha}(I)= \overline{\beta}_{i,\alpha}(J)+\overline{\beta}_{i,\alpha}(K)+\overline{\beta}_{i-1,\alpha}(J\cap
K).$$ Indeed, this is a Betti splitting and the pruning algorithm
provides a minimal free resolution, thus
$${\beta}_{i,\alpha}(I)= {\beta}_{i,\alpha}(J)+{\beta}_{i,\alpha}(K)+{\beta}_{i-1,\alpha}(J\cap
K).$$ By induction, we get a minimal free resolution for the ideals
$J$ and $K$ so we only have to control the intersection in order to
get the desired result. Recall that, using Lemma \ref{minimal}, we
only have to consider a
 minimal set of generators. In our case we have
$$J\cap K=\langle \underbrace{x_1x_2x_{n-1}x_n, \dots,
x_{n-4}x_{n-3}x_{n-1}x_n}_{J'},\underbrace{x_{n-2}x_{n-1}x_n}_{K'}\rangle.
$$
If we consider the decomposition $J\cap K=J'+K'$, we have a Betti splitting.
Namely, the pruning algorithm applied on the ideals $J'$ and $J'\cap
K'$ is equivalent to the one given for the path $\langle x_1x_2,
\dots, x_{n-4}x_{n-3}\rangle$, and we are done by induction.
\end{example}

\vskip 2mm

\begin{example}[$n$-cycles]
Let $I= \langle x_1x_2,x_2x_3,\dots,x_{n-1}x_n,x_nx_1 \rangle$ be the edge ideal of
an $n$-cycle. The decomposition $I=J+K$ with $J= \langle
x_1x_2,x_2x_3,\dots,x_{n-1}x_{n} \rangle$ and $K=\langle x_{n}x_1
\rangle$ is not a Betti splitting by \cite[Theorem 3.4]{FHV}.
Indeed, there is a pruning in the last step of Algorithm \ref{alg1}.

\vskip 2mm

It is more convenient to consider a splitting vertex . Namely,
consider $I=J+K$ with $J= \langle x_1x_2,x_2x_3,\dots,x_{n-2}x_{n-1}
\rangle$ and $K=\langle x_{n-1}x_{n}, x_{n}x_1 \rangle$. Since $J$ and $K$
are the ideal of an $(n-1)$-path and a $3$-path respectively, the pruning algorithm provides a
minimal free resolution as shown in Example \ref{ex_npath}. Therefore we have:
$$\overline{\beta}_{i,\alpha}(I)= {\beta}_{i,\alpha}(J)+{\beta}_{i,\alpha}(K)+\overline{\beta}_{i-1,\alpha}(J\cap
K).$$

\vskip 2mm
A minimal set of generators of $J\cap K$ is
$$\langle \underbrace{x_2x_3x_{n-1}x_{n},\dots,x_{n-4}x_{n-3}x_{n-1}x_{n},
x_{n-2}x_{n-1}x_{n}}_{J'}, \underbrace{
x_1x_2x_{n},x_1x_3x_{4}x_{n},\dots,x_{1}x_{n-3}x_{n-2}x_{n}}_{K'}
\rangle\,.$$
We have a pruned Betti splitting given by $J\cap K= J' + K'$. The
pruning algorithm gives a minimal free resolution for the ideals
$J'$ and $K'$ as we have seen when dealing with the case of paths. A
minimal set of generators for the intersection $J'\cap K'$ is
$$\langle x_1x_2x_3x_{n-1}x_n,x_1x_3x_{4}x_{n-1}x_n, \dots ,
x_1x_{n-3}x_{n-2}x_{n-1}x_n,x_{1}x_2x_{n-2}x_{n-1}x_n \rangle$$ but
the pruning algorithm applied to this ideal is equivalent to the one
for the cycle $\langle x_2x_3,x_3x_{4}, \dots ,$ 
$x_{n-3}x_{n-2},x_2x_{n-2} \rangle$, and we are done by induction.
\end{example}

\section*{Acknowledgement}
We would like to thank the referees for their meticulous comments that helped us to improve this manuscript.

\end{document}